\documentclass[11pt]{article}
\usepackage{epigamath}

\usepackage[notext]{kpfonts}
\usepackage{baskervald}

\setpapertype{A4}

\usepackage[english]{babel}

\usepackage[dvips]{graphicx}     

\removelength{1cm}

\title{Cartan geometries on complex manifolds of algebraic dimension zero}
\titlemark{Cartan geometries on complex manifolds of algebraic dimension zero}
\author{Indranil Biswas, Sorin Dumitrescu and Benjamin McKay}
\authoraddresses{
\authordata{Indranil Biswaw}{\firstname{Indranil} \lastname{Biswas}\\
\institution{School of Mathematics, Tata Institute of Fundamental
Research, Homi Bhabha Road, Mumbai 400005, India}\\
\email{indranil@math.tifr.res.in}} \\
\authordata{Sorin Dumitrescu}{\firstname{Sorin} \lastname{Dumitrescu}\\
\institution{Universit\'e C\^ote d'Azur, CNRS, LJAD}\\
\email{dumitres@univ-cotedazur.fr}}\\
\authordata{Benjamin McKay}{\firstname{Benjamin} \lastname{McKay}\\
\institution{University College Cork, Cork, Ireland}\\
\email{b.mckay@ucc.ie}}
}
\authormark{I. Biswas, S. Dumitrescu and B. McKay}
\date{\vspace{-5ex}} 
\journal{\'Epijournal de G\'eom\'etrie Alg\'ebrique} 
\acceptation{Received by the Editors on April 24, 2018, and in final form on August 11, 2019. \\ Accepted on October 25, 2019.}

\newcommand{\articlereference}{Volume 3 (2019), Article Nr. 19}

\usepackage{todonotes}

\newtheorem{theorem}{Theorem}[section]

\newtheorem{lemma}[theorem]{Lemma}

\newtheorem{definition}[theorem]{Definition}

\usepackage{float}

\numberwithin{equation}{section}

\makeatletter

\renewcommand{\p@equation}{\arabic{section}.\arabic{equation}\expandafter\@gobble}
\makeatother

\begin{document}


\maketitle



\begin{prelims}

\vspace{-0.55cm}

\def\abstractname{Abstract}
\abstract{We prove that every compact complex manifold of algebraic dimension zero bearing a holomorphic Cartan geometry of algebraic type 
must have infinite fundamental group. This generalizes the main Theorem in \cite{DM} where the same result was proved for the
special cases of holomorphic affine connections and holomorphic conformal structures.}

\keywords{Cartan geometry, almost homogeneous space, algebraic dimension, Killing vector field, semistability}

\MSCclass{53C56, 14J60, 32L05}


\languagesection{Fran\c{c}ais}{%

\vspace{-0.05cm}
{\bf Titre. G\'eom\'etries de Cartan sur les vari\'et\'es complexes de dimension alg\'ebrique nulle} \commentskip {\bf R\'esum\'e.} Nous montrons que toute vari\'et\'e complexe compacte de dimension alg\'ebrique nulle poss\'edant une g\'eom\'etrie de Cartan holomorphe de type alg\'ebrique doit avoir un groupe fondamental infini. Il s'agit d'une g\'en\'eralisation du th\'eor\`eme principal de \cite{DM} o\`u le m\^eme r\'esultat \'etait d\'emontr\'e dans le cas particulier des connexions affines holomorphes et des structures conformes holomorphes.}

\end{prelims}


\newpage

\setcounter{tocdepth}{1} \tableofcontents

\section{Introduction}

Let $G$ be connected complex Lie group and $H\, \subset\, G$ a closed complex Lie subgroup. A holomorphic Cartan geometry of 
type $(G,\, H)$ on a complex manifold $X$ is an infinitesimal structure on $X$ modeled on $G/H$ (details are in Section 
\ref{section: geometric structures}). Holomorphic affine connections, holomorphic projective connections and holomorphic 
conformal structures are among the important geometric examples of holomorphic Cartan geometries \cite{Sh}.

Contrary to the real setting, compact complex manifolds admitting a holomorphic Cartan geometry are
rather rare. All known 
examples of compact complex manifolds admitting a holomorphic Cartan geometry actually admit a {\it flat} holomorphic Cartan 
geometry. It should be clarified that the model of the flat Cartan geometry (see Definition \ref {cartan geom 
def}) could be different from the
given one. Indeed, compact complex parallelizable manifolds \cite{Wa} biholomorphic to quotients $G / 
\Gamma$, where $G$ is a complex semi-simple Lie group and $\Gamma \,\subset\, G$ a uniform lattice, are known to admit holomorphic 
affine connections, but no flat holomorphic affine connections \cite{Du2}.

Nevertheless, the following was conjectured in \cite{DM}: {\it Compact complex simply connected manifolds admitting a holomorphic 
Cartan geometry with model $(G,\,H)$ are biholomorphic to the model $G/H$.} (See
\cite[Section 6]{DM}.)

This question is open even for the very special case of holomorphic projective connections: accordingly to this conjecture, 
complex compact simply connected manifolds bearing a holomorphic projective connection should be complex projective spaces 
(endowed with the standard flat projective geometry); this is not known even for smooth complex projective varieties (except for 
the complex dimensions one, two \cite{KO, KO1} and three \cite{PR}). A special case of this conjecture is that compact complex 
simply connected manifolds $X$ do not admit any holomorphic affine connection. This is known to be true for K\"ahler manifolds; 
indeed, in this case the rational Chern classes of the holomorphic tangent bundle $TX$ must vanish, \cite{At}, and using Yau's 
theorem proving Calabi's conjecture it can be shown that $X$ admits a finite unramified cover which is a complex torus 
\cite{IKO}. According to the conjecture, the only simply connected compact complex manifold which admits a holomorphic conformal 
structure is the smooth nondegenerate quadric hypersurface in
complex projective space. However, this is not known even for smooth complex
projective varieties (except for the complex dimensions one, two \cite{KO2} and three \cite{PR1}). It is known that the
smooth quadric 
hypersurface admits precisely one holomorphic conformal structure (the standard one) \cite[Corollary 3]{BMc}.

Examples of non--K\"ahler complex compact manifolds of algebraic dimension zero admitting holomorphic Riemannian metrics 
\cite{Le, Gh} of constant sectional curvature
were constructed by Ghys in \cite{Gh}. If the sectional curvatures of a holomorphic Riemannian metric
are constant, then the associated Levi-Civita holomorphic affine connection is projectively flat.
They also admit the corresponding flat holomorphic conformal 
structure. The examples in \cite{Gh} are deformations of parallelizable manifolds covered by ${\rm SL}(2, \mathbb C)$.

In the direction of the conjecture, the following was proved in \cite[Theorem 1]{DM}:

{\it Compact complex simply connected manifolds with algebraic dimension zero admit neither holomorphic affine connections, nor 
holomorphic conformal structures.}

Generalizing the main theorem in \cite{DM} (Theorem 1), here we prove the following:

{\it Compact complex manifolds of algebraic dimension zero bearing a holomorphic Cartan geometry of algebraic type have infinite 
fundamental group} (Theorem \ref{main thm}).

\section{Cartan geometries}\label{section: geometric structures}

Let $X$ be a complex manifold, $G$ a complex connected Lie group and $H$ a closed complex Lie subgroup of $G$. The 
Lie algebras of the Lie groups $G$ and $H$ will be denoted by $\mathfrak{g}$ and $\mathfrak h$ respectively.

\begin{definition}\label{cartan geom def}
A {\it holomorphic Cartan geometry} $(P,\, \omega)$ on $X$ with model $(G,\,H)$ is a 
holomorphic principal (right) $H$-bundle
\begin{equation}\label{pi}
\pi\,:\, P \,\longrightarrow\, X
\end{equation}
endowed with a holomorphic $\mathfrak g$-valued
1-form $\omega$ satisfying the following three conditions:
\begin{enumerate}
\item $\omega_{p} \,:\, T_{p}P \,\longrightarrow\, \mathfrak g$ is a complex linear isomorphism for all $p \,\in\, P$,
 
\item the restriction of $\omega$ to every fiber of $\pi$ coincides with the left invariant Maurer-Cartan form of $H$, and
 
\item $(R_{h})^* \omega \,=\,{\rm Ad}(h)^{-1} \omega$, for all $h \,\in \,H$, where $R_{h}$ is the right action of $h$ on $P$
and ${\rm Ad}$ is the adjoint representation of $G$ on $\mathfrak g$.
\end{enumerate}
\end{definition}

\begin{definition}
The \emph{kernel} \(N\) of a model \((G,\,H)\) is the largest subgroup of \(H\) normal in \(G\); the kernel is clearly
a closed complex normal subgroup, and we denote its Lie algebra by \(\mathfrak{n}\).
A model \((G,\,H)\) is \emph{effective} if its kernel is \(N\,=\,\left\{1\right\}\). In other words,
\((G,\,H)\) is effective if $H$ does not contain any nontrivial subgroup normal in $G$.
A Cartan geometry is \emph{effective} if its model is effective.
Any Cartan geometry \((P,\,\omega)\) has an \emph{induced effective} Cartan geometry
\((\overline{P},\,\overline{\omega})\) where \(\overline{P}\,=\,P/N\) and \(\overline{\omega}\)
is the unique \(\mathfrak{g}/\mathfrak{n}\)-valued 1-form on \(\overline{P}\) which pulls back,
via the natural quotient map \(P\,\longrightarrow\, \overline{P}\),
to the $\mathfrak{g}/\mathfrak{n}$-valued one-form given by \(\omega\).
\end{definition}

\begin{definition}
If $\text{Ad}(H/N)\, \subset\, \text{GL}({\mathfrak g}/\mathfrak{n})$ is an algebraic subgroup, then
the holomorphic Cartan geometries with model $(G,\,H)$ are said to be of {\it algebraic type}.
\end{definition}
\hspace{1,4em}Holomorphic affine connections, holomorphic projective connections, and holomorphic conformal structures are among the examples of effective holomorphic Cartan geometries of algebraic type \cite{Pe,Sh}. 

Notice that Definition \ref{cartan geom def} implies, in particular, that the complex dimension of the homogeneous model
space $G/H$ coincides with the complex dimension $n$ of the manifold $X$.

Let us give some important examples.

A holomorphic affine connection corresponds to the model (${\mathbb C}^n\rtimes\text{GL}(n, 
{\mathbb C}),\, \text{GL}(n, {\mathbb C})$) of the complex affine space acted on by the group 
${\mathbb C}^n \rtimes\text{GL}(n, {\mathbb C})$ of complex affine transformations.

A holomorphic projective connection is a Cartan geometry with model $(\text{PGL}(n+1,{\mathbb C}), 
\,Q)$, where $\text{PGL}(n+1,{\mathbb C})$ is the complex projective group acting on the complex 
projective space ${\mathbb C}P^n$ and where $Q\, \subset\, \text{PGL}(n+1,{\mathbb C})$ is the maximal 
parabolic subgroup that fixes a given point in ${\mathbb C}P^n$.

Holomorphic conformal structures are modeled on the quadric \(z_0^2+ z_1 ^2 + \ldots 
+z_{n+1}^2\,=\, 0\) in ${\mathbb C}P^{n+1}$. Here $G$ is the subgroup ${\rm PO}(n+2, \mathbb{C})$ of 
the complex projective group $\text{PGL}(n+2,{\mathbb C})$ preserving the above quadric, while $H$ is the 
stabilizer of a given point in the quadric. For more details about those Cartan geometries and for 
many other examples the reader may consult \cite{Sh}.

\begin{definition} 
Let $X$ be equipped with a holomorphic Cartan geometry $(P,\, \omega)$ of type $(G,\, H)$.
A (local) biholomorphism between two open subsets $U$ and $V$ in $X$ is a (local) {\it automorphism} of
the Cartan geometry $(P,\, 
\omega)$ if it lifts to a holomorphic principal $H$-bundle isomorphism between $\pi^{-1}(U)$ and $\pi^{-1}(V)$
that preserves the form $\omega$; here $\pi$ is the projection in \eqref{pi}.

A (local) holomorphic vector field on $X$ is a (local) {\it Killing field} of $(P,\, \omega)$ if its (local) flow acts by (local) 
automorphisms of $(P,\, \omega)$.
 
The Cartan geometry $(P,\, \omega)$ is called {\it locally homogeneous} on an open subset $U \,\subset\, X$ if the tangent
space $T_uX$ is spanned by local Killing fields of $(P,\, \omega)$ for every point $u \,\in\, U$.
\end{definition}

The curvature of $(P,\, \omega)$ is defined as a $2$-form $\Omega$ on $P$, with values in $\mathfrak{g}$, given by the 
following formula
\begin{equation}\label{om2}
\Omega(Y,\,Z)\,=\,d \omega (Y,\,Z)- \lbrack \omega (Y), \,\omega (Z) \rbrack _{\mathfrak{g}}
\end{equation}
for all 
local holomorphic vector fields $Y,\, Z$ on $P$. When at least one of $Y$ and $Z$ is a vertical vector field
(meaning lies in the kernel of the differential $d\pi$ of $\pi$ in \eqref{pi}), then $\Omega(Y,\,Z)\,=\, 0$ (see \cite[5.3.10]{Sh}). From this it follows that $\Omega$ is basic and it descends on $X$.

More precisely we have the following classical interpretation of $\Omega$. Let
\begin{equation}\label{e2}
P_G\, :=\, P \times^H G \, \stackrel{\pi_G}{\longrightarrow}\, X
\end{equation}
be the holomorphic principal $G$--bundle over $X$ obtained by extending the
structure group of the holomorphic principal $H$--bundle
$P$ using the inclusion of $H$ in $G$. Recall that $P_G$ is the
quotient of $P \times G$ where two points $(p_1,\, g_1),\, (p_2,\, g_2)\, \in\,
P\times G$ are identified if there is an element $h\, \in\, H$ such that
$p_2\,=\, p_1h$ and $g_2\,=\, h^{-1}g_1$. The projection $\pi_G$ in \eqref{e2} is induced
by the map $P \times G\, \longrightarrow\, X$, $(p,\, g)\,\longmapsto\, \pi (p)$.
The action of $G$ on $P_G$ is induced by the
action of $G$ on $P \times G$ given by the right--translation action of $G$ on itself.

Consider the adjoint vector bundle ${\rm ad}(P_G)\,=\, P_G\times^G {\mathfrak g}$ of 
$P_G$. We recall that ${\rm ad}(P_G)$ is the quotient of $P_G\times
\mathfrak g$ where two points $(c_1,\, v_1)$ and $(c_2,\, v_2)$ are identified if there
is an element $g\, \in\, G$ such that
$c_2\, =\, c_1g$ and $v_2$ is the image of $v_1$ through the automorphism of the Lie algebra
$\mathfrak g$ defined by automorphism of $G$ given by $z \,\longmapsto\, g^{-1}zg$.

Starting with $(P,\, \omega)$ there is a natural Ehresmann connection on the principal bundle $P_G$ given by the following
$\mathfrak g$-valued form on $P \times G$:

$$\widetilde{\omega}(p,g)\,=\, Ad(g^{-1})\pi_1^*(\omega) + \pi_2^* (\omega_G)\, ,$$
where $\pi_1$ and $\pi_2$ are the projections of $P \times G$
on the first and the second factor respectively, while $\omega_G$ is the left-invariant Maurer-Cartan form on $G$.

The above form $\widetilde{\omega}$ is invariant by the previous $H$-action on $P \times G$ and it vanishes
when restricted to the fibers of the fibration $P \times G \,\longrightarrow\, P_G$.
This implies that $\widetilde{\omega}$ is basic and it descends on $P_G$. Let us denote also by $\widetilde{\omega}$
this one-form on $P_G$ with values in $\mathfrak g$: it is an Ehresmann
connection on the principal bundle $P_G$ (see \cite{BD}, Proposition 3.4).

Recall that the curvature of the connection $\widetilde{\omega}$ on the principal bundle $P_G$ is a tensor of the following type:
\begin{equation}\label{curv}
\text{Curv}(\widetilde{\omega})\, \in\, H^0(X,\, \text{ad}(P_G)\otimes \Omega^2_X)\, .
\end{equation}

Moreover, the curvature $\Omega$ of $\omega$ in \eqref{om2}
vanishes if and only if $\text{Curv}(\widetilde{\omega})$ vanishes.
If $\Omega\,=\, 0$, then $\omega$ produces local isomorphisms of $X$ to the homogeneous space $G/H$. This way we get
a developing map from the universal cover of $X$ to $G/H$, which is a local biholomorphism (see for example \cite{Sh}
or \cite[p.~9]{BD}).

\section{Cartan geometries on manifolds of algebraic dimension zero}\label{simply connected}

As a preparation for the proof of Theorem \ref{main thm}, we first prove the following Lemma 
\ref{symetry} which is an adaptation, for Cartan geometries, of \cite[Proposition 2]{DM} on rigid geometric structures.

Recall that a compact complex manifold is of algebraic dimension zero if it does not admit any nonconstant meromorphic 
function \cite{Ue}. These manifolds are far from being algebraic. Indeed, a manifold is bimeromorphic with an algebraic 
manifold if and only if its field of meromorphic functions separates points \cite{Mo, Ue}.

\begin{lemma}\label{symetry} Let $X$ be a compact complex simply connected manifold of 
algebraic dimension zero endowed with a holomorphic effective Cartan geometry $(P,\, \omega)$. 
Then the connected component of the automorphism group of the identity of $(P, \,\omega)$ is a connected complex 
abelian Lie group $L$ acting on $X$ with an open dense orbit. The open dense orbit is the 
complement of an anticanonical divisor. Moreover, the natural map from
$L$ to this open dense orbit, constructed by fixing a point of this orbit, is a biholomorphism.
Furthermore, $L$ is covered by $({\mathbb C}^*)^{n}$.
\end{lemma}

\begin{proof}
By Theorem 1.2 in \cite{Du} (see also \cite{D1}) the Cartan geometry $(P,\, \omega)$ is locally 
homogeneous on an open dense subset in $X$. Recall that on simply connected manifolds, local Killing 
vector fields of Cartan geometries extend to all of the manifold $X$ \cite{Am, No, Gr, Pe}.
Therefore, it 
follows that there exists a Lie algebra $\mathfrak{a}$ of globally defined Killing vector fields 
on $X$ which span $TX$ at the generic point. Let us fix a basis $\{X_1,\, \cdots,\, X_k \}$ 
of $\mathfrak{a}$ and consider the generalized Cartan geometry $(P, \omega, \mathfrak{a})$ (in the 
sense of Definition 4.11 in \cite{Pe}) which is a juxtaposition of $(P, \omega)$ with the family 
of vector fields $\{X_1,\, \cdots,\, X_k \}$. The proof of Theorem 1.2 in \cite{Du} applies to the 
generalized geometry $(P, \,\omega,\, \mathfrak{a})$ which must also be locally homogeneous on an open 
dense set. Notice that local Killing vector fields of $(P, \omega, \mathfrak{a})$ are restrictions 
of elements in $\mathfrak{a}$ which commute with the vector fields $X_i$. It follows that local 
vector fields of $(P,\, \omega,\, \mathfrak{a})$ are elements in the centralizer $\mathfrak{a'}$ 
of $\mathfrak{a}$, and $\mathfrak{a'}$ is transitive on an open dense subset. Since $\mathfrak{a'}$ is 
commutative, the action is simply transitive on the open dense subset (note that $\mathfrak{a'}$ has the 
same dimension as that of $X$). A basis of $\mathfrak{a'}$ is a family of commuting vector fields spanning 
$TX$ at the generic point. Since the vector fields $X_i$ commute with elements in this basis, they 
must be a linear combination of elements in $\mathfrak{a'}$ with constant coefficients (on the 
open dense subset and hence on all of $X$). In particular, we have $X_i \,\in\, \mathfrak{a'}$, which implies 
that $\mathfrak{a}\,=\,\mathfrak{a'}$ is an abelian algebra; this
also implies that $k$ is the complex dimension of 
$X$. In particular, the corresponding connected Lie group $L$ (the connected component of the 
identity of the automorphism group of $(P,\, \omega)$) is an abelian group of the same dimension 
as that of $X$ acting with an open dense orbit. This open dense orbit is the complement of the vanishing set $S$ of 
the holomorphic section $X_1 \bigwedge \cdots \bigwedge X_k$ of the anticanonical bundle. 

Since $L$ is abelian, any \(\ell_0 \,\in\, L\) stabilizing some point \(x_0\) of the open 
orbit stabilizes all points \(\ell x_0\) for \(\ell \,\in\, L\): \(\ell_0 x_0\,=\,x_0\), so \(\ell_0 
\ell x_0 \,= \,\ell \ell_0 x_0\,=\,\ell x_0\). By density of the open dense orbit, \(\ell_0\) 
stabilizes all points of \(X\). So the stabilizer of any point in the open orbit is trivial, 
i.e., \(X\) is an \(L\)-equivariant compactification of \(L\).

The following arguments which prove that
$L$ is covered by $({\mathbb C}^*)^{n}$ are borrowed from \cite{DM} (Proposition 2). Let us define nontrivial meromorphic 1-forms \(\omega_i\) on \(X\) by
$$
\omega_i(\xi) \,=\,
\frac{X_1 \wedge X_2 \wedge \cdots \wedge X_{i-1} \wedge \xi \wedge X_{i+1} \wedge \cdots \wedge X_n}
{X_1 \wedge X_2 \wedge \cdots \wedge X_n}.
$$

Notice that the previous meromorphic 1-forms $\omega_i$ are $L$-invariant and are holomorphic when 
restricted to $X \setminus S$. Since $L$ is abelian the classical Lie-Cartan formula implies 
that the forms $\omega_i$ are closed (on the open dense orbit $X \setminus S$ and hence on all of 
$X$). The indefinite integral of any nonzero closed meromorphic 1-form is a nonconstant 
meromorphic function on some covering space of the complement of the simple poles of the 
1-form. Since $X$ is simply connected and does not admit nontrivial meromorphic functions, 
every nonzero meromorphic 1-form on $X$ must have a simple pole on some component of the 
complement $S$ of the open dense orbit of $L$.

Define $$\Delta \,=\, H_1(X \setminus S,\, \mathbb Z)\,=\,H_1(L, \,\mathbb Z)\, .$$ Then 
$\Delta$ identifies with a discrete subgroup in the Lie algebra $\mathfrak{a}$ of $L$ and $L$ 
is isomorphic to the quotient $\mathfrak{a} / \Delta$.

The restriction to $X \setminus S$ of the $L$-invariant meromorphic $1$-forms $\omega_i$ 
identify with a basis of $\mathfrak{a}^*$.

Pair any \(\gamma \,\in\, \Delta,\, \omega \,\in\, \mathfrak{a}^*\) as follows:
\begin{equation}\label{pair}
\gamma,\, \omega\, \longmapsto\, \int_{\gamma} \omega \,\in\, \mathbb C{}\, .
\end{equation}
Assume that for a given \(\omega\) this pairing vanishes for every \(\gamma\). Then \(\omega\) 
integrates around each component of \(S\) to define a meromorphic function on $X$. This 
meromorphic function must be constant and therefore we have \(\omega\,=\,0\). Consequently, the pairing
in \eqref{pair} is nondegenerate, and hence $\Delta \,\subset\, \mathfrak{a}$ spans $\mathfrak{a}$ over 
$\mathbb C$. This implies that $L\, =\, \mathfrak{a} / \Delta$ is covered by $({\mathbb C}^*)^{n}$. 
\end{proof}

Recall that a large and important class of non-K\"ahler smooth equivariant compactifications of complex abelian Lie groups 
(namely the so-called LMBV manifolds) was constructed in \cite {BM,Mer} (see also \cite{PUV}).

\section{Equivariant compactifications of abelian groups and stability}

In this section we prove the main result of the article (Theorem \ref{main thm}).

Let us first recall the notion of slope stability for holomorphic
vector bundles complex manifolds which will be used in the sequel.

Let $X$ be a compact complex manifold of complex dimension $n$. Let us fix a Gauduchon metric $g$ on it. Recall that the 
$(1,1)$-form $\alpha_g$ associated to a Gauduchon metric $g$ satisfies the equation
$$\partial \overline\partial \alpha_g^{n-1}\,=\,0\, .$$ By a result of Gauduchon, any hermitian metric on $X$ is conformally
equivalent to a Gauduchon metric (satisfying 
the above equation) which is uniquely defined up to a positive constant \cite{Ga}.

The degree of a torsionfree coherent analytic sheaf $F$ on $X$ is defined to be
\begin{equation}\label{deg}
\text{deg}(F)\,:=\, 
\frac{\sqrt{-1}}{2\pi}\int_M K(\det F)\wedge \alpha_g^{n-1}\, \in\, {\mathbb R}\, ,
\end{equation}
where $\det F$ is the determinant line 
bundle for $F$ \cite[Ch.~V, \S~6]{Ko} (or Definition 1.34 in \cite{Br}) and $K$ is the curvature for a hermitian connection on 
$\det F$ compatible with $\overline \partial_{\det F}$. This degree is independent of the hermitian metric since any two such 
curvature forms differ by a $\partial \overline\partial$-exact $2$--form on $X$:
$$
\int_X (\partial \overline\partial u)\wedge \alpha_g^{n-1}\,=\, -\int_X u \wedge\partial \overline\partial \alpha_g^{n-1}\,=\, 0\, .
$$

This notion of degree is independent of the choice of the 
hermitian connection, but depends on the Gauduchon metric (it is not a topological invariant). In the particular case where $d \alpha_g =0$ we find the classical degree for K\"ahler manifolds (and in this case the degree is a topological invariant).

Define $$\mu(F)\,:=\,
\frac{\text{deg}(F)}{\text{rank}(F)}\, \in\, {\mathbb R}\, ,$$ which is called the
{\it slope} of $F$ (with respect to $g$).

A torsionfree coherent analytic sheaf $F$ on $X$ is called \textit{stable} (respectively,
\textit{semi-stable}) if for every coherent analytic subsheaf $W \,\subset\, F$
such the $\text{rank}(W)\, \in\, [1\, ,\text{rank}(F)-1]$, the inequality
$\mu(W) \, <\, \mu(F)$ (respectively, $\mu(W) \, \leq\, \mu(F)$)
holds (see \cite[p.~44, Definition 1.4.3]{LT}, \cite[Ch.~V, \S~7]{Ko}).

The rest of the section is devoted to the proof of the main theorem of the article:

\begin{theorem}\label{main thm}
Compact complex manifolds with algebraic dimension zero bearing a holomorphic Cartan 
geometry of algebraic type have infinite fundamental group.
\end{theorem}

\begin{proof}
Assume, by contradiction, that a complex manifold $X$ of algebraic dimension zero admits a holomorphic 
Cartan geometry $(P,\, \omega)$ and has finite fundamental group. Considering the universal cover of $X$ and pulling back
the Cartan geometry $(P,\, \omega)$ to it, we would assume that $X$ is simply connected.
Replace the Cartan geometry by the induced effective Cartan geometry, so we can assume that it is effective.

By Lemma \ref{symetry}, the maximal connected subgroup of the automorphism group of
$(P,\, \omega)$ is a complex connected 
abelian Lie group $L$ acting on $X$ with an open dense orbit (it is the complement of an 
anticanonical divisor on $X$). The complement of this anticanonical divisor is biholomorphic to $L$ 
(see Lemma \ref{symetry} again).

Let us now prove the following lemma describing the geometry of simply connected smooth equivariant compactifications of complex abelian Lie groups.

\begin{lemma}\label{main lemma}
Let $X$ be a simply connected smooth equivariant compactification of a complex abelian Lie group $L$. Then the following
four hold:
\begin{enumerate}
\item There is no nontrivial holomorphic $k$-form on $X$, for $k \,\geq\, 1$.

\item Any holomorphic line bundle $E_1$ over $X$ admitting a holomorphic connection is holomorphically trivial.

\item If $E$ is a holomorphic vector bundle over $X$ admitting a holomorphic connection, then $E$ is semi-stable of degree zero.

\item If $F$ is a coherent analytic subsheaf of the holomorphic tangent bundle $TX$, such that $TX/ F$ is nonzero and torsionfree, then 
the quotient ${\rm deg}(TX/ F) \,>\,0$.
\end{enumerate}
\end{lemma}

\begin{proof}[{Proof of Lemma \ref{main lemma}}] 
(1) Let us first consider the special case of $k\,=\,1$. Let $\eta$ be a holomorphic one-form on $X$. Let $\{X_1,\, \cdots,\, X_n \}$ be a family of commuting holomorphic vector fields on $X$ which span
$TX$ at the generic point. By the Lie-Cartan formula we get that for any $i,\,j \,\in\, \{1,\, \cdots,\, n\}$:
\begin{equation}\label{dc}
d \eta (X_i,\, X_j) \,=\, X_i \cdot \eta (X_j) - X_j \cdot \eta (X_i)- \eta ( \lbrack X_i, X_j \rbrack )\,=\,0\, .
\end{equation}
Since any holomorphic function on $X$ is a constant one and $[X_i,\, X_j]\,=\, 0$ for all $i,\, j$, it follows from
\eqref{dc} that the form
$\eta$ is closed. Now since $X$ is simply connected, $\eta$ coincides with the differential of a global holomorphic
function $u$ on $X$. Since $X$ is compact, $u$ is constant and hence $\eta \,=\,0$.

Now consider a $k$-form $\eta$ on $X$ with $k \,>\,1$. For any $i_1,\, i_2,\, \cdots ,\,i_{k-1} \,\in\, \{1,\,
\cdots,\, n\}$, the one-form on $X$
$$v \, \longmapsto\, \eta(X{_{i_1}},\, \cdots,\, X_{i_{k-1}},\, v)$$ vanishes by the previous proof. Consequently, $\eta$ must vanish.

(2) Let $\nabla$ be a holomorphic connection on a holomorphic line bundle $E_1$. The curvature ${\rm Curv}(\nabla)$ of
$\nabla$ is a holomorphic two-form on $X$. Hence ${\rm Curv}(\nabla)\,=\, 0$ by (1). So $\nabla$ is flat. This implies
that $E_1$ is the trivial holomorphic line bundle, because $X$ is simply connected.

(3) Let $\nabla$ be a holomorphic connection of $E$. The determinant line bundle ${\det}(E)$ of $E$
has a holomorphic connection induced by $\nabla$. So ${\det}(E)$ is holomorphically trivial by (2).
Hence the degree of $E$ is zero (see \eqref{deg}).

Assume, by contradiction, that $E$ is not semi-stable. Then there exists a maximal semi-stable coherent
analytic subsheaf $W$ (the 
smallest nonzero term of the Harder--Narasimhan filtration of $E$) \cite[pp.~14--15, Theorem 1.3.1]{HL}. The maximal
semistable subsheaf $W$ has the following property:
\begin{equation}\label{qp}
H^0(X,\, \text{Hom}(W,\, E/W))\, =\, 0\, .
\end{equation}
Indeed, this follows immediately from the facts that the slopes of the graded pieces for the Harder--Narasimhan filtration of
$E/W$ are strictly less than $\mu(W)$, and there is no nonzero homomorphism from a semistable sheaf to another semistable sheaf
of strictly smaller slope.

Now consider the second fundamental form $$S\,:\, W
\,\longrightarrow\, (E/W)\otimes \Omega^1_X$$
of $W$ for the holomorphic connection $\nabla$, which associates to any locally defined holomorphic section $s$ of $W$ the
projection, to $(E/W)\otimes \Omega^1_X$, of $\nabla (s)$ (which is a section of $E\otimes \Omega^1_X$). 
For any global holomorphic vector field $X_i$ on $X$ we have the homomorphism
$W\, \longrightarrow \, E/W$ defined by $w\, \longmapsto\,S(w)(X_i)$. From \eqref{qp} we know that this homomorphism vanishes
identically. Since the vector fields $X_i$ span $TX$ at the generic point, we have $S\,=\, 0$. Therefore, the connection
$\nabla$ preserves $W$. In other words, $\nabla$ induces a holomorphic connection on $W$.
This implies that the coherent analytic sheaf $W$ is locally free (see Lemma 4.5 in \cite{BD}). 
Since $W$ admits a holomorphic
connection, from (2) we have ${\rm deg}(W)\,=\, 0$. This contradicts the assumption that $W$ destabilizes $E$.

Consequently, $E$ is semistable.

(4) Take $F\, \subset\, TX$ as in (4). Let $TX/F$ be of rank $r \geq 1$. Consider the projection
$q\, :\, TX \,\longrightarrow\, TX /F$ and the associated homomorphism
\begin{equation}\label{de}
\widehat{q}\, :\, \bigwedge\nolimits^r TX \,\longrightarrow\, {\det}(TX/F)\, .
\end{equation}
Note that $q$ produces a homomorphism
$\bigwedge^r TX \,\longrightarrow\, {\det}(TX/F)$ over the subset $U\, \subset\, X$ where the torsionfree coherent analytic
sheaf $TX /F$ is locally free; now this homomorphism extends to entire $X$ by Hartogs' theorem, because ${\det}(TX/F)$
is locally free and the complex codimension of $X\setminus U$ is at least two.

Choose $i_1, \,\cdots,\, i_r \,\in\, \{ 1,\, 2,\, \cdots,\, n\}$ for which $\{ q(X_{i_1}), \,\cdots,\,q(X_{i_r})\}$ generate
$TX /F$ at the generic point. Then $q(X_{i_1}) \bigwedge \cdots \bigwedge q(X_{i_r})$ defines a holomorphic 
section of ${\det}(TX/F)$ which is not identically zero; we shall denote this
section of ${\det}(TX/F)$ by $\sigma$. So we have ${\rm deg}(TX/F)\, \geq\, 0$, because ${\det}(TX/F)$ admits a
nonzero holomorphic section namely $\sigma$.

Assume by contradiction ${\rm deg}(TX/F) \,=\, 0$. Then the divisor ${\rm div}(\sigma)\, \subset\, X$ of the
above section $\sigma$ 
is the zero divisor. Indeed, the degree of $TX/F$ is the volume of ${\rm div}(\sigma)$ with respect to the given
Gauduchon metric (see \cite[Proposition 5.23]{Br}). Consequently, the line bundle ${\det}(TX/F)$ is holomorphically trivializable.
Once a holomorphic
trivialization of ${\det}(TX/F)$ is fixed, the homomorphism $\widehat{q}$ in \eqref{de} becomes a holomorphic
$1$-form on $X$ which is not identically zero. This is in contradiction with (1).
This proves that ${\rm deg}(TX/ F) \,>\,0$.
\end{proof}

Continuing with the proof of Theorem \ref{main thm}, let us now consider the curvature of the Cartan geometry $(P,\, \omega)$
$$
\text{Curv}(\widetilde{\omega})\, \in\, H^0(X,\, \text{ad}(P_G)\otimes \Omega^2_X)
$$
constructed in \eqref{curv}. When contracted with any global holomorphic vector field $X_i$ on $X$,
this curvature form $\text{Curv}(\widetilde{\omega})$ produces a holomorphic
homomorphism
\begin{equation}\label{vp}
\varphi(X_i) \, :\, TX \, \longrightarrow\, \text{ad}(P_G)\, ,\ \ v\, \longmapsto\, \text{Curv}(\widetilde{\omega})(X_i,\, v)\, .
\end{equation}
The holomorphic connection $\widetilde{\omega}$ on $P_G$ induces a holomorphic connection on the vector bundle
$\text{ad}{(P_G)}$ associated to the principal $G$--bundle $P_G$. Hence from Lemma \ref{main lemma}(3) we know that
the vector bundle $\text{ad}{(P_G)}$ is semistable of degree zero.
Consequently, from Lemma \ref{main lemma}(3) it follows immediately that
$$
H^0(X,\, \text{Hom}(TX,\, \text{ad}(P_G))\,=\, 0\, ;
$$
indeed, the image of a nonzero homomorphism $TX \, \longrightarrow\, \text{ad}(P_G)$ would contradict the semistability condition
for $\text{ad}(P_G)$. In particular, the homomorphism $\varphi(X_i)$ in \eqref{vp} vanishes identically.

Since the vector fields $X_i$ generate the tangent bundle of $X$ over a nonempty open subset of $X$, from the vanishing of
the homomorphisms $\varphi(X_i)$ we conclude that $\text{Curv}(\widetilde{\omega})\,=\,0$. Consequently,
the Cartan geometry $(P,\, \omega)$ is flat.

Since $X$ is simply connected, the developing map of the Cartan geometry is a holomorphic biholomorphism between $X$ and
$G/H$. But $G/H$ is algebraic and $X$ has algebraic dimension zero: a contradiction. This completes
the proof of Theorem \ref{main thm}.
\end{proof} 

Recall that a {\it generalized Cartan geometry} in the sense of \cite{BD} is given by a pair $(P,\, \omega)$ satisfying 
conditions (2) and (3) in Definition \ref{cartan geom def}, while the homomorphism in Definition \ref{cartan geom def}(1) is 
not required to be an isomorphism anymore. Consequently, $X$ and $G/H$ do not necessarily have the same dimension.
 
The proof of Theorem \ref{main thm} shows that on simply connected smooth equivariant compactifications $X$ of complex abelian 
Lie groups, all (generalized) holomorphic Cartan geometries in the sense of \cite{BD} are flat (notice that the condition on the 
algebraic dimension of $X$ is not needed for this part of the proof; neither the condition on the algebraicity of $G/H$). In 
conclusion all (generalized) holomorphic Cartan geometries on $X$ are given by holomorphic maps $X \,\longrightarrow\, G/H$.

\section*{Acknowledgements}

The authors are very grateful to Laurent Meersseman for many enlightening discussions on the subject. 
We thank the referee for helpful comments.
Indranil Biswas is partially supported by a J. C. Bose Fellowship. Sorin Dumitrescu wishes to 
thank T.I.F.R. Mumbai for hospitality. This research was supported in part by the International 
Centre for Theoretical Sciences (ICTS) during a visit for participating in the program - Analytic 
and Algebraic Geometry (Code: ICTS/aag2018/03).


\vspace{0.5cm}

\bibliographymark{References}

\providecommand{\bysame}{\leavevmode\hbox to3em{\hrulefill}\thinspace}
\providecommand{\arXiv}[2][]{\href{https://arxiv.org/abs/#2}{arXiv:#1#2}}
\providecommand{\MR}{\relax\ifhmode\unskip\space\fi MR }
\providecommand{\MRhref}[2]{%
  \href{http://www.ams.org/mathscinet-getitem?mr=#1}{#2}
}
\providecommand{\href}[2]{#2}

\end{document}